\documentclass[11pt]{amsart}
\usepackage{amsmath, amssymb, amsthm}
\usepackage{hyperref}
\usepackage{tikz}
\usepackage{subfigure}
\usepackage{ytableau}
\usepackage{array}
\usepackage{amsfonts}
\usepackage{mathtools}

\theoremstyle{plain}
\newtheorem{theorem}{Theorem}[section]
\newtheorem{lemma}[theorem]{Lemma}

\newtheorem{proposition}[theorem]{Proposition}
\newtheorem{conjecture}[theorem]{Conjecture}

\newtheorem{remark}[]{Remark}
\newtheorem{claim}[]{Claim}

\numberwithin{equation}{section}

\begin{document}
	\title[Zeros in The Character Table of The Symmetric Group
	]{Lower Bound for The Number of Zeros in The Character Table of The Symmetric Group} 
	\author{JAYANTA BARMAN}
	\address{JAYANTA BARMAN\\ Department of Mathematics \\
		Indian Institute of Technology Kharagpur \\
		Kharagpur-721302, India.} 
	\email{b1999jayanta@gmail.com}
	
	\author[Kamalakshya Mahatab]{Kamalakshya Mahatab$^\dagger$}
	\address{Kamalakshya Mahatab\\ Department of Mathematics \\
		Indian Institute of Technology Kharagpur \\
		Kharagpur-721302, India.} 
	\email{kamalakshya@maths.iitkgp.ac.in}
	
	\thanks{$^\dagger$Corresponding author: \texttt{kamalakshya@maths.iitkgp.ac.in}}
	
	
	\begin{abstract}
		For any two partitions $\lambda$ and $\mu$ of a positive integer $N$, let $\chi_{\lambda}(\mu)$ be the value of the irreducible character of the symmetric group $S_{N}$ associated with $\lambda$, evaluated at the conjugacy class of elements whose cycle type is determined by $\mu$. Let $Z(N)$ be the number of zeros in the character table of $S_N$, and $Z_{t}(N)$ be defined as  
		$$
		Z_{t}(N):= \#\{(\lambda,\mu): \chi_{\lambda}(\mu) = 0 \; \text{with $\lambda$ a $t$-core}\}.
		$$  
		We prove
		$$
		Z(N) \ge \frac{2\, p(N)^{2}}{\log N} \left(1+O\left(\frac{1}{\sqrt{\log N}}\right)\right),
		$$
		where $p(N)$ denotes the number of partitions of $N$. We also give explicit lower bounds for $Z_t(N)$ in various ranges of $t$. 
	\end{abstract}
	
	\maketitle
	\section{Introduction} 
	For any two partitions $\lambda$ and $\mu$ of a positive integer $N$, let $\chi_{\lambda}(\mu)$ denote the value of the irreducible character of the symmetric group $S_{N}$ associated with $\lambda$, evaluated in the conjugacy class of elements whose cycle type is determined by $\mu$. By the Murnaghan-Nakayama rule \cite{fulton1991j}, it is known that irreducible characters are integer-valued functions, and the number of irreducible characters of $S_{N}$ is equal to $p(N)$, the number of partitions of $N$. In 1918, Hardy and Ramanujan \cite{hardy1918asymptotic} used the circle method to prove the celebrated asymptotic formula for $p(N)$: 
	\begin{align*}
		p(N) \sim \frac{1}{4N\sqrt{3}} \exp\bigg(\frac{2\pi}{\sqrt{6}} \sqrt{N} \bigg).
	\end{align*}
	A more usable and explicit version of the above result, proved by  Rademacher \cite{rademacher1938partition}, gives:
	\begin{align}\label{eq-hardy}
		p(N) = \frac{1}{4N\sqrt{3}} \exp\bigg(\frac{2\pi}{\sqrt{6}} \sqrt{N} \bigg) \bigg(1+O(N^{-\frac{1}{2}})\bigg),
	\end{align}
	which we will use frequently.
	In this article, we study the zeros of the character values. Although linear characters never take the value zero, Burnside's classical result \cite{burnside1904arithmetical} establishes that every non-linear irreducible character must vanish at some group element. A statistical estimate proved by Miller \cite{miller2014probability} states that if one chooses an irreducible character of $S_{N}$ uniformly at random and selects a random element from $S_{N}$ uniformly, then the probability that the character value is zero approaches $1$ as $N \to \infty$. However, this result does not estimate the number of zeros in the character table of $S_N$ since the character values are distributed over the conjugacy classes, rather than individual elements of $S_N$.  Note that the character table of $S_N$ has $p(N)^{2}$ entries. Let $Z(N)$ be the number of zeros in the character table of the symmetric group $S_N$. 
	Miller \cite{miller2014probability, miller2019parity} introduced the problem of determining the asymptotic behavior of $Z(N)$. Due to the rapid growth of $p(N)$, computation of $Z(N)$ is challenging. To tackle this problem, Miller and Scheinerman \cite{miller2025large} conducted a large-scale Monte Carlo simulation to determine the density of zeros for large values of $N$, leading to the following conjecture:
	\begin{conjecture}\label{conjec}
		$\frac{Z(N)}{p(N)^{2}} \sim  \frac{2}{\log N}$ as $N \to \infty$.
	\end{conjecture}
	Peluse \cite{peluse2020even} proved that the proportion of zeros in the character table of $S_N$ is at least $A / \log N$ for some positive constant $A$. 
	Later, in \cite{peluse2025divisibility}, Peluse and Soundararajan remarked that $Z(N) \geq \frac{2 p(N)^2}{\log N}$, without any details.
	In this article, we prove the above lower bound.
	\begin{theorem}\label{theorem-1.2} Let $N$ be a large positive integer. Then
		$$
		Z(N) \ge \frac{2 p(N)^{2}}{\log N}  \left(1+O\left(\frac{1}{\sqrt{\log N}}\right)\right).
		$$ 
	\end{theorem}		
	If Conjecture \ref{conjec} is true, then the main term of our lower bound is optimal.

	Our proof uses the following inequality (see Theorem~\ref{theorem-2.1}) based on the Murnaghan–Nakayama rule:
	\begin{align}
		Z(N) &\geq 
		\label{eq-(1.1)}
		\sum_{t=1}^{N} c_t(N)p_t(N-t),
	\end{align} 
	where $p_t(N)$ denotes the number of partitions of $N$ into parts of size at most $t$.
	Here, we attempt to obtain an exact order of the above sum. However, this requires asymptotic estimates for $c_t(N)$ and $p_t(N-t)$. We take the asymptotic formula for $c_t(N)$ from the recent work of Tyler \cite{tyler2026asymptotics} and the asymptotic formula for $p_t(N-t)$ from Erd\" os and Lehner \cite{erdos1941distribution}. We will see later in the proof that Tyler's formula plays an important role as it gives an asymptotic bound for $c_t(N)$ when $t$ and $N$ both vary. In the later part of the proof, we treat the above sum in different ranges of $t$ to obtain an optimal bound.
	
	We may also restrict our investigation to the number of zeros in a strip of the character table. In particular, we may consider only the rows where $\lambda$ is a $t$-core. Define
	$$
	Z_{t}(N) := \#\{(\lambda,\mu): \chi_{\lambda}(\mu) = 0 \; \text{with $\lambda$ a $t$-core}\}.
	$$	
	McSpirit and Ono \cite{mcspirit2023zeros} proved the following result for fixed primes $t \geq 5$:
	$$
	Z_{t}(N)\gg_t N^{\frac{t-5}{2}} \exp\left(\pi \sqrt{2N/3}\right), \ N\rightarrow \infty.
	$$

	We obtain the following lower bounds for $Z_{t}(N)$ as both $N, t \rightarrow \infty$. This gives an explicit version of McSpirit and Ono's result \cite{mcspirit2023zeros} for all $t$.
	\begin{theorem}\label{theorem-1.3} 
		Let $N\ge 100$ be a positive integer, and $t\leq N$. Then we have the following results:
		\newline
		(i) For all $0<\epsilon<1$, if $6\leq t \leq \frac{2\pi\sqrt{2N}}{\sqrt{(1+\epsilon)\log N}}$, then 
		$$
		Z_{t}(N) \geq R_{t}(N)p(N)\left(1+O\left(\frac{t}{\sqrt{N}}+t^{-\epsilon}\right)\right),
		$$ 
		where 
		$$
		R_{t}(N)=\frac{(4\pi e)^{\frac{t-1}{2}}(t-1)}{\sqrt{4\pi}(t^{2}-t)^{\frac{t}{2}}} \left(N+\frac{t^{2}-1}{24}\right)^{\frac{t-3}{2}}.
		$$
		(ii)  There exist $f(N)\sim \frac{\sqrt{24N}}{\sqrt{\frac{6}{\pi}-1}}$, such that if  
		$\frac{2\pi\sqrt{2N}}{\sqrt{\log N}}< t < f(N)$, then
		$$  
		Z_{t}(N)\ge Q_{t}(N)p(N)\left(1+O\left(\frac{t}{\sqrt{N}}\right)\right),
		$$  
		where 
		$$  
		Q_{t}(N)=\frac{2\sqrt{\pi}\exp\left(\frac{t-1}{2}-1.00873te^{-2\pi}\right)\left(\frac{\pi}{6}(24N+t^{2}-1)\right)^{\frac{t-3}{2}}}{t^{t-1}}.  
		$$
		(iii) For $t \ge f(N)$, where $f(N)$ is defined in part (ii), we have
		$$
		Z_{t}(N)\geq \frac{p(N)^{2}}{\exp\left(1.00873t\exp\left(-\frac{ t(t-1)}{2\left(N+\frac{t^{2}-1}{24}\right)}\right)+\frac{2\pi}{\sqrt{6}} \frac{t}{\sqrt{N-t}+\sqrt{N}}\right)} \left(1+O\left(\frac{t}{N}\right)\right).
		$$  
		
	\end{theorem}
	We believe that the above bounds, in Theorem~\ref{theorem-1.2} and Theorem~\ref{theorem-1.3}, can be generalized to  Weyl groups and wreath products of symmetric groups \cite{dong2023almost}. 
	
	\section{Acknowledgments}
	We extend our heartfelt thanks to Prof. Ben Kane for his insightful and valuable suggestions. J. Barman is deeply thankful to the University Grants Commission (UGC), India, for their invaluable support through the Fellowship Programme. K. Mahatab is supported by the DST INSPIRE Faculty Award Program (grant no. DST/INSPIRE/04/2019/002586) and ARG-MATRICS program (grant no. ANRF/ARGM/2025/002540/MTR).
	\section{Preliminaries}
	The hook $h$ associated with a box $b$ in the Young diagram of a partition $\lambda$ includes the box $b$ itself, along with all the boxes located directly to the right of $b$ and those directly below $b$. The length of the hook $h$, denoted by $\ell(h)$, is the total number of boxes contained within the hook $h$. For example, in the Young diagram of $\lambda = (4,2,1)$ shown below, each box is labeled with its corresponding hook length.
	$$
	\ytableausetup{centertableaux}
	\begin{ytableau}
		6 &  4 & 2 & 1
		\\
		3 & 1 \\
		1\\
	\end{ytableau}
	$$
	$$
	\text{Figure 1.  Hook-lengths for }\lambda=(4,2,1)
	$$
	The height of a hook $h$, denoted by $\text{ht}(h)$, is defined as one less than the total number of rows in the Young diagram of $\lambda$ that contain a box belonging to $h$. Each hook is associated with a border strip (also called a skew hook), denoted by $\text{bs}(h)$, which is the continuous boundary region of the Young diagram extending from the rightmost box of $h$ to its bottommost box. Removing this border strip yields a smaller Young diagram.
	\par
	A partition is called a $t$-core if none of the hook lengths in its Young diagram are divisible by $t$. For example, as illustrated in Figure 1, the partition $(4, 2, 1)$ is a $5$-core.
	
	We now recall the Murnaghan–Nakayama rule, a classical result used to compute the character values of irreducible representations of the symmetric group $S_N$.
	\begin{theorem}[The Murnaghan-Nakayama rule]\label{theorem-2.1}
		Let $N$ and $t$ be positive integers such that $t \leq N$. Consider $\sigma \in S_{N}$, expressed as $\sigma = \tau \cdot \rho$, where $\rho$ is a $t$-cycle, and $\tau$ is a permutation in $S_{N}$ whose support is disjoint from that of $\rho$.  Then
		$$
		\chi_{\lambda}(\sigma)=\sum_{\substack{h \in \lambda \\ \ell(h)=t}}(-1)^{\mathrm{ht}(h)} \chi_{\lambda \backslash \mathrm{bs}(h)}(\tau) .
		$$  
	\end{theorem}
	The notation $\lambda \setminus \operatorname{bs}(h)$ refers to the partition of $N-t$ obtained by removing the border strip $\operatorname{bs}(h)$ from the Young diagram of $\lambda$. Additionally, $\chi_{\lambda \setminus \operatorname{bs}(h)}(\tau)$ denotes the character value of the irreducible representation of $S_{N-t}$ corresponding to the partition $\lambda \setminus \operatorname{bs}(h)$, evaluated at the conjugacy class of $\tau$. We may obtain the following result using the Murnaghan-Nakayama rule, which gives a sufficient condition for the character value to be zero.
	\begin{lemma}[{\cite[Lemma 2.2]{peluse2020even}}]\label{lemma-2.2}
		Let $\lambda$ and $\mu$ be two partitions of $N$. If $\mu$ has a part of size $t$  and $\lambda$ is a $t$-core, then $\chi_{\lambda}(\mu)=0.$
	\end{lemma}
	
	If we consider the partitions $\mu$ having the largest part $t$ and the $t$-core partitions $\lambda$, then by the above lemma $\chi_{\lambda}(\mu)=0.$ This set of pairs $(\lambda, \mu)$ gives $c_t(N)p_t(N-t)$ zeros. We also notice that the set of partitions $\mu$ of $N$ having the largest part $t_1$ are disjoint from the set of partitions of $N$ having the largest part $t_2$. So the number of zeros in the character table is at least 
	\[\sum_{t=1}^{N} c_t(N)p_t(N-t),\]
	which gives (\ref{eq-(1.1)}).
	\newline
	Next, we simplify Tyler's formula for $c_t(N)$, which requires the following notations.

	The Dedekind eta function $\eta(z)$ is defined by 
	$$
	\eta(z)=\exp\left(\frac{\pi iz}{12}\right)\prod_{n=1}^{\infty}(1-\exp(2\pi inz)),
	$$
	where $z=x+iy$, $y>0$.
	In  \cite{tyler2026asymptotics}, Tyler defines the following functions:
	\begin{align*}
		&\mu_{k}(z)=-\frac{z^{k+1}}{2\pi i} \left(\frac{d}{dz}\right)^{k} \log\eta(z),\quad \text{for $k\ge 1$}\\
		&\text{and}\quad f_{t}(z)=\frac{\eta(tz)^{t}}{\eta(z)}.
	\end{align*}
	To approximate the Dedekind eta function for large $y$, we will use the following result.
	\begin{lemma}\label{lemma-2.3}
		For each $y\ge \frac{\sqrt{3}}{2}$, there exist a $v$ with $1<v<1.00873$ such that
		$$
		\eta(iy)=\exp\left(-\frac{\pi y}{12}-v e^{-2\pi y}\right).
		$$
	\end{lemma}
	\begin{proof}
		From the definition of $\eta(z)$,
		$$
		\log \eta(iy)=-\frac{\pi y}{12}-\sum_{n=1}^{\infty}\frac{\sigma(n)}{n}\exp(-2\pi ny).
		$$
		Using the above formula and proceeding as in the proof of Lemma 3.2 of \cite{tyler2026asymptotics}, we have
		\begin{align}\label{eq-valpha}
			\exp(-2\pi y)<\sum_{n=1}^{\infty}\frac{\sigma(n)}{n}\exp(-2\pi ny)<\sum_{n=1}^{\infty}n\exp(-2\pi ny)=\frac{\exp(-2\pi y)}{\left(1-\exp(-2\pi y)\right)^{2}}.  
		\end{align}
		Now we obtain an upper bound of the above choosing $y=\frac{\sqrt{3}}{2}$ and write
		\begin{align*}
			\exp(-2\pi y)<\sum_{n=1}^{\infty}\frac{\sigma(n)}{n}\exp(-2\pi ny)<\frac{\exp(-2\pi y)}{\left(1-e^{-\sqrt{3}\pi}\right)^{2}}<1.00873e^{-2\pi y}.
		\end{align*}
		This proves our lemma.
	\end{proof}   
	For small $y>0$, we will use the functional equation for $\eta(z)$ as below. 
	\begin{lemma}\label{lemma-2.4}
		For each $0<y<1$, there exist a $v$ with $1<v<1.00873$ such that
		$$
		\eta(iy) = y^{-\frac{1}{2}} \exp\left(-\frac{\pi}{12y} - v e^{-\frac{2\pi}{y}} \right).
		$$
	\end{lemma}
	\begin{proof}
		By the modular transformation formula,
		$$
		\eta\left(-\frac{1}{z}\right) = \sqrt{-iz}\, \eta(z),
		$$
		which holds for all $z$ in the upper half-plane. Applying this with $z = iy$, we obtain
		$$
		\eta(iy) = y^{-\frac{1}{2}}\, \eta\left( \frac{i}{y} \right).
		$$
		Using Lemma \ref{lemma-2.3}, we obtain our required result.
	\end{proof}
	
	\begin{lemma}[{\cite[Lemma 4.2]{tyler2026asymptotics}}]\label{lemma-2.5}
		Let $\mu_2$ and $0<y\le\frac{1}{10}$ be defined as before. Then for any positive integer $t>1$, the following inequalities hold:
		\newline
		$(i)$ For $ty<1$, we have
		$$  
		\frac{2\sqrt{\pi}}{\sqrt{y(t-1)}}<\frac{1}{\sqrt{\mu_{2}(iy)-\mu_{2}(ity)}}<\frac{2\sqrt{2\pi}}{\sqrt{y(t-1)}}.  
		$$
		$(ii)$ For $ty\ge 1$, we have
		$$
		\sqrt{12}<\frac{1}{\sqrt{\mu_{2}(iy)-\mu_{2}(ity)}}<\sqrt{16}.
		$$
	\end{lemma}
	Next, we recall the main theorem for $c_{t}(N)$ from Tyler's paper \cite{tyler2026asymptotics}.
	\begin{theorem}[{\cite[Theorem 1.4]{tyler2026asymptotics}}]\label{TylerThm}
		Assume that $N\ge 0$ and $t\ge 6.$  
		\newline
		(i) The equation
		\begin{align}\label{eq-Ty1}
			\frac{\mu_1(ity)-\mu_1(iy)}{y^{2}}=N+\frac{t^{2}-1}{24}  
		\end{align}
		admits a unique solution $y>0$, which satisfies
		\begin{align}\label{eq-Ty2}
			\frac{t-1}{4\pi\left(N+\frac{t^{2}-1}{24}\right)}<y<\frac{1}{\frac{3}{\pi}+\sqrt{24N-1+\frac{9}{\pi^{2}}}}.    
		\end{align}
		(ii) With this value of $y$, we have
		\begin{align}\label{eq-Ty3}
			c_{t}(N) = \frac{y^{\frac{3}{2}} \exp\left(2\pi y \left(N+\frac{t^{2}-1}{24}\right)\right) (\eta(ity))^{t}}
			{\sqrt{\mu_{2}(iy)-\mu_{2}(ity)}\,\eta(iy)} 
			\left(1+O\left(\frac{1}{\text{min}(t,\sqrt{N})}\right)\right).
		\end{align}
	\end{theorem}
	Now, we simplify Tyler's bound \cite{tyler2026asymptotics} for $c_{t}(N)$ in different ranges of $t$.
	\begin{proposition}\label{proposition-3.2} 
		Let $N\ge 100$ be a positive integer, and $t\leq N$. 
		\newline
		(i)  For all $0<\epsilon<1$, if $6\leq t \leq \frac{2\pi\sqrt{2N}}{\sqrt{(1+\epsilon)\log N}}$, then
		$$
		c_{t}(N)=\frac{(4\pi e)^{\frac{t-1}{2}}(t-1)}{\sqrt{4\pi}(t^{2}-t)^{\frac{t}{2}}} \left(N+\frac{t^{2}-1}
		{24}\right)^{\frac{t-3}{2}}(1+O(t^{-\epsilon})).
		$$
		(ii) There exists $f(N)\sim \frac{\sqrt{24N}}{\sqrt{\frac{6}{\pi}-1}}$ such that, if  
		$\frac{2\pi\sqrt{2N}}{\sqrt{\log N}}< t < f(N)$, then
		$$  
		c_{t}(N)\ge \frac{2\sqrt{\pi}\exp\left(\frac{t-1}{2}-1.00873te^{-2\pi}\right)\left(\frac{\pi}{6}(24N+t^{2}-1)\right)^{\frac{t-3}{2}}}{t^{t-1}} \left(1+O\left(t^{-1}\right)\right).  
		$$
		(iii) For $t \ge f(N)$, where $f(N)$ is defined in part (ii), we have
		\begin{align*}
			c_{t}(N)\ge p(N)\exp\left(-1.00873t\exp\left(-\frac{ t(t-1)}{2\left(N+\frac{t^{2}-1}{24}\right)}\right)\right)\left(1+O\left(N^{-\frac{1}{2}}\right)\right).
		\end{align*}
		\newline
		(iv) For $\frac{\sqrt{6}}{2\pi}\sqrt{N}\log N<t$ and $N\ge 300000$, we have
		\begin{equation*}
			c_{t}(N) \ge p(N) \exp\left(-t\exp\left(-\frac{\pi t}{\sqrt{6N}}\right)\right) \left(1+O\left(N^{-\frac{1}{2}}(\log N)^{4}\right)\right).
		\end{equation*}
	\end{proposition}
	\begin{remark}
		Part $(i)$ of the above proposition gives an asymptotic bound for $c_t(N)$. In parts $(ii)$ and $(iii)$, we separate different ranges of $t$
		according to whether $ty<1$ or $ty\ge 1$.
	\end{remark}
	\begin{remark}
		Part (iv) of Proposition \ref{proposition-3.2} plays a central role in the proof of Theorem \ref{theorem-1.2} as it provides a sharper bound than part (iii) for large $N$. Here, we assume $N>300000>e^{4\pi}$ to ensure that $ty\ge \frac{\sqrt{N}}{2\pi}\sqrt{N}\log N\frac{1}{\sqrt{24N}}\ge1$. 
	\end{remark}

	\begin{proof} $(i)$ From equation (1.7) of \cite{tyler2026asymptotics}, when $6\le t\le \frac{2\pi\sqrt{2N}}{\sqrt{(1+\epsilon)\log N}}$ and $0< \epsilon<1$, we have
		\begin{equation}\label{eq-(2.1)}
			c_{t}(N)=\frac{(2\pi)^{\frac{t-1}{2}}}{t^{\frac{t}{2}}\Gamma\left(\frac{t-1}{2}\right)}\left(N+\frac{t^{2}-1}{24}\right)^{\frac{t-3}{2}}(1+O(t^{-\epsilon})).  
		\end{equation}
		Using Stirling's approximation,  
		$$  
		\Gamma\left(\frac{t-1}{2}\right)=\sqrt{\frac{4\pi}{t-1}}\left(\frac{t-1}{2e}\right)^{\frac{t-1}{2}}\left(1+O\left(t^{-1}\right)\right).  
		$$  
		Substituting the above in (\ref{eq-(2.1)}), we obtain  
		$$  
		c_{t}(N)=\frac{(4\pi e)^{\frac{t-1}{2}}(t-1)}{\sqrt{4\pi}(t^{2}-t)^{\frac{t}{2}}}\left(N+\frac{t^{2}-1}{24}\right)^{\frac{t-3}{2}}(1+O(t^{-\epsilon})).  
		$$ 
		$(ii)$
		From (\ref{eq-Ty2}) and the specified range of $t$, we observe that $ty<1$. By Lemma \ref{lemma-2.4}, we obtain
		\begin{align}\label{eq-part2}
			\eta(ity)\ge (ty)^{-\frac{1}{2}}\exp\left(-\frac{\pi}{12ty}-1.00873e^{-2\pi}\right).
		\end{align}	
		To obtain the following bound for $c_t(N)$, we consider the formula (\ref{eq-Ty3}), and put in the lower bound for $\eta(ity)$ from (\ref{eq-part2}), upper bound for $\eta(iy)$ from Lemma \ref{lemma-2.4}, and upper bound of $\mu_2(iy)-\mu_2(ity)$ from Lemma \ref{lemma-2.5}
		\begin{equation*}
			c_{t}(N)\ge\frac{2\sqrt{\pi}\exp\left(2\pi y \left(N+\frac{t^{2}-1}{24}\right)-1.00873te^{-2\pi}+e^{-\frac{2\pi}{y}}\right)}{t^{\frac{t+1}{2}}y^{\frac{t-3}{2}}} \left(1+O\left(t^{-1}\right)\right).  
		\end{equation*}
		In the above bound, we make $y$ explicit in terms of $t$ and $N$ by minimizing the right-hand side at $ y= \frac{t-1}{4\pi\left(N+\frac{t^{2}-1}{24}\right)}$, lower end point of the interval of $y$ in (\ref{eq-Ty2}), and obtain
		$$  
		c_{t}(N)\ge \frac{2\sqrt{\pi}\exp\left(\frac{t-1}{2}-1.00873te^{-2\pi}\right)\left(\frac{\pi}{6}(24N+t^{2}-1)\right)^{\frac{t-3}{2}}}{t^{t-1}} \left(1+O\left(t^{-1}\right)\right).  
		$$  
		$(iii)$ Again by (\ref{eq-Ty2}), we see that $ty\ge 1$ for the specified range of $t$. From Lemma \ref{lemma-2.3}, we can write
		\begin{align}\label{eq-part3}
			\eta(ity)\ge \exp\left(-\frac{\pi ty}{12}-1.00873e^{-2\pi ty}\right).
		\end{align}
		Similar to the proof of $(ii)$, we consider the asymptotic formula of $c_t(N)$ from Theorem \ref{TylerThm} $(ii)$ and use the bounds of $\eta(ity)$, $\eta(iy)$, and $\mu_2(iy)-\mu_2(ity)$ from equation (\ref{eq-part3}), Lemma \ref{lemma-2.4} and Lemma \ref{lemma-2.5} respectively. This gives
		\begin{align*}
			&c_{t}(N) \geq \sqrt{12} y^{2} \exp\left( 2\pi Ny
			+ \frac{\pi}{12y}-\frac{\pi y}{12} - 1.00873t \exp(-2\pi yt) +e^{-\frac{2\pi}{y}}\right)\\&\hspace{9cm}\times\left(1+O\left(N^{-\frac{1}{2}}\right)\right)\\
			&\geq \sqrt{12} y^{2} \exp\left(\left( \sqrt{2\pi Ny}
			-\sqrt{\frac{\pi}{12y}}\right)^{2}+\frac{2\pi\sqrt{N}}{\sqrt{6}}- 1.00873t \exp(-2\pi yt) +e^{-\frac{2\pi}{y}}\right)\\&\hspace{9cm}\times\left(1+O\left(N^{-\frac{1}{2}}\right)\right).
		\end{align*}
		To minimize the above expression, we set the squared term in the exponent equal to zero.
		From (\ref{eq-Ty2}), we have
		\begin{align*}
			& \exp\left(- 1.00873t \exp(-2\pi yt)\right)\ge \exp\left(-1.00873t\exp\left(-\frac{t(t-1)}{2\left(N+\frac{t^{2}-1}{24}\right)}\right)\right).
		\end{align*}
		Also, note 
		\begin{align*}
			\exp\left(e^{-\frac{2\pi}{y}}\right)=1+O\left(N^{-\frac{1}{2}}\right).
		\end{align*}
		Using these simplifications and the asymptotic result for the partition function (\ref{eq-hardy}), we have
		\begin{align*}
			c_{t}(N)\ge p(N)\exp\left(-1.00873t\exp\left(-\frac{t(t-1)}{2\left(N+\frac{t^{2}-1}{24}\right)}\right)\right)\left(1+O\left(N^{-\frac{1}{2}}\right)\right).
		\end{align*}
		\newline
		$(iv)$ From Lemma \ref{lemma-2.3}, we have
		\begin{equation}\label{eq-alpha}
			\eta(ity)=\exp\left(-\frac{\pi ty}{12}- v e^{-2\pi ty}\right), 
		\end{equation} where $1<v<\alpha=\frac{1}{(1-\exp(-2\pi yt))^{2}}$. The value of $\alpha$ comes from equation (\ref{eq-valpha}).
		\newline
		Using the explicit expression of $\mu_{1}$ for large\footnote{We refer to the imaginary part as large when it lies in
			$[1,\infty)$, and as small when it lies in $(0,1)$.} imaginary parts (see \cite[(4.14) and (4.16)]{tyler2026asymptotics}), we deduce 
		\begin{equation*}
			\mu_{1}(ity)=-\sum_{n=1}^{\infty}t^{2}y^{2}\sigma(n)\exp(-2\pi nty)+\frac{t^{2}y^{2}}{24},   
		\end{equation*}
		and for small imaginary parts
		\begin{equation*}
			\mu_{1}(iy)=\sum_{n=1}^{\infty}\sigma(n)\exp\left(-\frac{2\pi n}{y}\right) -\frac{1}{24}+\frac{y}{4\pi}.  
		\end{equation*} 
		Substituting the above expressions for $\mu_{1}$ in (\ref{eq-Ty1}), we obtain
		\begin{align}\label{eq-saddleeq}
			y^{2}\left(N-\frac{1}{24}\right)+\frac{y}{4\pi}-\left(\frac{1}{24}-K_1-K_2\right)=0,        
		\end{align}
		where
		\begin{align*}
			K_1=\sum_{n=1}^{\infty}t^{2}y^{2}\sigma(n)\exp(-2\pi nty)\,\,\text{and  } K_2=\sum_{n=1}^{\infty}\sigma(n)\exp\left(-\frac{2\pi n}{y}\right).
		\end{align*}
		We may consider (\ref{eq-saddleeq}) as a quadratic equation and solve 
		\begin{align*}
			y=-\frac{1}{8\pi\left(N-\frac{1}{24}\right)} +\frac{1}{\sqrt{24\left(N-\frac{1}{24}\right)}}\sqrt{1+\frac{3}{8\pi^{2}\left(N-\frac{1}{24}\right)}-24(K_1+K_2)}.   
		\end{align*}
		Substituting the expression for $y$ into $K_{1}$ \footnote{Note that  $K_1=\sum_{n=1}^{\infty}t^{2}y^{2}\sigma(n)\exp(-2\pi nty)\le \frac{(\log N)^{2}}{16\pi^{2}\sqrt{N}}+O\left(\frac{(\log N)^{2}}{N}\right)$ for the specified range of $t$. This estimate is used in the derivation of (\ref{eq-soly}).} and $K_{2}$ \footnote{We also note that $K_2=\sum_{n=1}^{\infty}\sigma(n)\exp\left(-\frac{2\pi n}{y}\right)=O(N^{-2}).$} in the given range of $t$, and simplifying further, we obtain 
		\begin{equation}\label{eq-soly}
			y=\frac{1}{\sqrt{24N}}+\frac{C_{1}(N,t)}{N}+ \frac{C_{2}}{N}+O\left(\frac{(\log N)^{2}}{N^{\frac{3}{2}}}\right)  
		\end{equation}
		is the solution of (\ref{eq-saddleeq}),  where $0<|C_1(N,t)|\le C_3(\log N)^{2}$ and, $C_2$ and $C_3$ are constants independent of $N$. We may also compute
		\begin{align}\label{eq-soly-inverse}
			y^{-1}=\sqrt{24N}\left(1-\frac{C_1(N,t)\sqrt{24}}{\sqrt{N}}-\frac{C_2\sqrt{24}}{\sqrt{N}}+O\left(\frac{(\log N)^{2}}{N}\right)\right).   
		\end{align}
		We next use the asymptotic formula for $c_t(N)$ given in
		Theorem \ref{TylerThm} $(ii)$ and combine it with the bounds for
		$\eta(ity)$, $\eta(iy)$, and $\mu_2(iy)-\mu_2(ity)$ from
		equation (\ref{eq-alpha}), Lemma \ref{lemma-2.4}, and Lemma \ref{lemma-2.5},
		respectively. This yields 
		\begin{align}\label{eq-initial} 
			\notag
			c_{t}(N) &\geq \sqrt{12} y^{2} \exp\left( y \left(2\pi N - \frac{\pi}{12}\right)  
			+ \frac{\pi}{12y} - tv \exp(-2\pi yt) +e^{-\frac{2\pi}{y}}\right)\\
			&\hspace{7.2cm}\times\left(1+O\left(N^{-\frac{1}{2}}\right)\right).  
		\end{align}  
		Now we simplify the terms inside the exponential in the following steps.
		Note from (\ref{eq-soly}) and (\ref{eq-soly-inverse})
		\begin{align}\label{eq-yy-inverse}
			2\pi Ny+\frac{\pi}{12y}=\frac{2\pi}{\sqrt{6}}\sqrt{N}+O\left(\frac{(\log N)^{2}}{N^{\frac{1}{2}}}\right).   
		\end{align}
		We also have
		\begin{align}\label{eq-error}
			-tv\exp(-2\pi yt)&=-tv\exp(-2\pi y_0t)+tv\exp(-2\pi y_0t)\left(1-\exp(-2\pi tE)\right)
		\end{align} 
		where
		\begin{align*}
			y_0=\frac{1}{\sqrt{24N}}\,\,\text{and  }E= \frac{C_{1}(N,t)}{N}+ \frac{C_{2}}{N}+O\left(\frac{(\log N)^{2}}{N^{\frac{3}{2}}}\right).  
		\end{align*}
		We take \(\exp\) of (\ref{eq-error}) and 
		substitute the value of $y$ from (\ref{eq-soly}) to get
		\begin{align}\label{eq-expy}
			\notag
			&\exp\left(-tv\exp(-2\pi yt)\right)\\
			\notag 
			&=\exp\left(-tv\exp\left(-2\pi y_0 t\right)\right)\exp\big(t v \exp(-2\pi y_{0} t)\big(1-\exp(-2\pi tE)\big)\big)\\
			&=\exp\left(-tv\exp\left(-\frac{\pi t}{\sqrt{6N}}\right)\right)\left(1+O\left(N^{-\frac{1}{2}}(\log N)^{4}\right)\right).
		\end{align}
		Here we used the fact that in the range $\frac{\sqrt{6}}{2\pi}\sqrt{N}\log N < t \le N$, the error attains its maximum at $t=\frac{\sqrt{6}}{2\pi}\sqrt{N}\log N$, which gives
		\[\exp\big(t v \exp(-2\pi y_{0} t)\big(1-\exp(-2\pi tE)\big)\big)=1+O\left(N^{-\frac{1}{2}}(\log N)^{4}\right).\]
		
		Now, using simplifications of (\ref{eq-soly}), (\ref{eq-yy-inverse}) and (\ref{eq-expy}) in (\ref{eq-initial}), and using asymptotic formula for $p(N)$ from (\ref{eq-hardy}), we obtain
		\begin{align*}
			c_t(N)&\ge\frac{\exp\left(\frac{2\pi}{\sqrt{6}}\sqrt{N}\right)}{4\sqrt{3}N}\exp\left(-tv\exp\left(-\frac{\pi t}{\sqrt{6N}}\right)\right)\left(1+O\left(N^{-\frac{1}{2}}(\log N)^{4}\right)\right)\\
			& =p(N)\exp\left(-tv\exp\left(-\frac{\pi t}{\sqrt{6N}}\right)\right) \left(1+O\left(N^{-\frac{1}{2}}(\log N)^{4}\right)\right). 
		\end{align*}
		From (\ref{eq-alpha}), $1<v<\alpha=\frac{1}{(1-\exp(-2\pi yt))^{2}}$. One can easily check that for the given range of $t$, $\alpha=1+O\left(N^{-\frac{1}{2}}\right)$.
		Consequently,
		\begin{equation*}
			c_{t}(N) \ge p(N) \exp\left(-t\exp\left(-\frac{\pi t}{\sqrt{6N}}\right)\right) \left(1+O\left(N^{-\frac{1}{2}}(\log N)^{4}\right)\right).   
		\end{equation*} 
		
	\end{proof}	
	In the next lemma, we present a sketch of the proof of the Erd\" os and Lehner's result for $p_t(N)$ \cite{erdos1941distribution}, with an explicit error term, which is an addition to the original paper.
	\begin{lemma}[Erd\"os, Lehner]\label{lemma-3.4}
		Let $p_{t}(N)$ be the number of partitions of $N$ in which no summands exceed $t$. Then, for  
		$
		t=C^{-1}\sqrt{N}\log N+x\sqrt{N},
		$
		we have  
		$$  
		p_{t}(N)=p(N)\exp\left(-\frac{2}{C}e^{-\frac{1}{2}Cx}\right)\left(1+O\left(\frac{(\log N +x)^{2}}{N^{\frac{1}{2}}}\right)\right),
		$$ 
		where	$C=2\pi/\sqrt{6}$ and $x\ll N^{\frac{1}{4}-\gamma},\ \gamma>0$.
	\end{lemma}  
	\begin{proof}  
		Erd\"{o}s and Lehner \cite{erdos1941distribution} expressed $p_t(N)$ as follows  
		\begin{align*}  
			p_{t}(N) &= p(N)-\sum_{1\le r\le N-t}p(N-(t+r))\\
			&+\sum_{\substack{0<r_{1}<r_{2}\\1<r_{1}+r_{2}\le N-2t} }p(N-(t+r_{1})-(t+r_{2}))\\  
			&\quad - \sum_{\substack{0<r_{1}<r_{2}<r_{3}\\1<r_{1}+r_{2}+r_{3}\le N-3t} }p(N-(t+r_{1})-(t+r_{2})-(t+r_{3})) -\cdots \\  
			&= p(N)(1-S_{1}+S_{2}-S_{3}+\cdots).  
		\end{align*}  
		We may observe  
		$$  
		1-S_{1}+S_{2}-\cdots-S_{2k-1} \leq \frac{p_{t}(N)}{p(N)} \leq 1-S_{1}+S_{2}-\cdots+S_{2k}.  
		$$  
		Further
		$$  
		\begin{aligned}  
			S_{1} &= \frac{1}{p(N)}\sum_{1\le r\le N-t}p(N-(t+r))\\  
			&= \frac{1}{p(N)}\sum_{r\le N^{\frac{3}{5}}}p(N-(t+r))+\frac{1}{p(N)}\sum_{r>N^{\frac{3}{5}}}p(N-(t+r))\\
			&=I_{1}+I_{2}.  
		\end{aligned}  
		$$ 
		We simplify the above sum using the formula (\ref{eq-hardy}). It can be easily verified that $I_2$ tends to zero as $N\to \infty$. For $I_1$, note 
		$$  
		\sum_{r\le N^{\frac{3}{5}}}\frac{N}{N-t-r}\exp\left(C\sqrt{N-t-r}-C\sqrt{N}\right)\left(1+O\left(N^{-\frac{1}{2}}\right)\right).  
		$$  
		Since 
		$t=C^{-1}\sqrt{N}\log N+x\sqrt{N}$,  we approximate $\sqrt{N-t-r} = \sqrt{N}-\frac{1}{2\sqrt{N}}(t+r)+O\left(\frac{t^{2}}{N^{\frac{3}{2}}}\right)$. As $x\ll N^{\frac{1}{4}-\gamma}$,  we have $\exp\left(O\left(\frac{t^{2}}{N^{\frac{3}{2}}}\right)\right)=1+O\left(\frac{t^{2}}{N^{\frac{3}{2}}}\right)$. Now we simplify $I_1$ as follows
		$$\begin{aligned}
			I_{1}&=\sum_{r\le N^{\frac{3}{5}}}\exp\left(-\frac{C(t+r)}{2\sqrt{N}}\right)\left(1+O\left(\frac{t^{2}}{N^{\frac{3}{2}}}\right)\right)\\
			&=N^{-\frac{1}{2}}\exp\left(-\frac{Cx}{2}\right)\sum_{1\le r\le N^{\frac{3}{5}}}\exp\left(-\frac{CrN^{-\frac{1}{2}}}{2}\right)\left(1+O\left(\frac{(\log N +x)^{2}}{N^{\frac{1}{2}}}\right)\right)\\
			&=\frac{2}{C}\exp\left(-\frac{1}{2}Cx\right)\left(1+O\left(\frac{(\log N +x)^{2}}{N^{\frac{1}{2}}}\right)\right).
		\end{aligned}
		$$
		Therefore,
		$$  
		S_{1} = \frac{2}{C}\exp\left(-\frac{1}{2}Cx\right)\left(1+O\left(\frac{(\log N +x)^{2}}{N^{\frac{1}{2}}}\right)\right).  
		$$ 
		Similarly, we find  
		$$  
		S_{k} = \frac{1}{k!}\left(\frac{2}{C}\exp\left(-\frac{1}{2}Cx\right)\right)^{k}\left(1+O\left(\frac{(\log N +x)^{2}}{N^{\frac{1}{2}}}\right)\right).  
		$$  
		This gives 
		$$  
		p_{t}(N)=p(N)\exp\left(-\frac{2}{C}e^{-\frac{1}{2}Cx}\right)\left(1+O\left(\frac{(\log N +x)^{2}}{N^{\frac{1}{2}}}\right)\right).  
		$$  
	\end{proof}
	\section{Proof of Theorem \ref{theorem-1.2} and \ref{theorem-1.3}}
	We now proceed to prove our main theorem. 
	\begin{proof}[Proof of Theorem \ref{theorem-1.2}]
		From (\ref{eq-(1.1)}), we have
			\begin{align}
				\notag
				Z(N) \ge\; &\sum_{t=1}^{N} c_{t}(N)\, p_{t}(N-t) \\
				\notag
				\ge\; &\sum_{t=T_1+1}^{T_2} c_{t}(N)\, p_{t}(N-t) 
				+ \sum_{t=T_2+1}^{N} c_{t}(N)\, p_{t}(N-t) \\
				\label{eq-(A)}
				=\; &A_{1} + A_{2}.
			\end{align}
		Here, we have chosen
		\begin{equation}\label{eq-T1T2}
			\begin{aligned}
				T_{1}&=\frac{\sqrt{6}}{2\pi}\sqrt{N}(\log N)\left(1+\frac{1}{2B}\right)\,\, \,\,\text{and  }
				T_2&= \frac{\sqrt{6}}{2\pi} \sqrt{N} (\log N) \left( 1 + \frac{1}{B} \right),
			\end{aligned}
		\end{equation}
		where $B$ is defined by the relation 
		$$N^{\frac{1}{2B}} =  \frac{\sqrt{6}}{2\pi} \log N.$$
		\newline
		Before we obtain bounds for $A_1$ and $A_2$, we need the following estimates of $c_t(N)$ and $p_t(N-t)$. By part $(iv)$ of Proposition \ref{proposition-3.2}, 
		\begin{align}
			&c_t(N) \ge\;  p(N) \exp\left(-t \exp\left(-\frac{\pi t}{\sqrt{6N}}\right)\right)
			\left(1 + O\left(N^{-\frac12}(\log N)^{4}\right)\right) \label{eq-(4.1)}\\
			=\; & p(N) \left( \sum_{j=0}^{\infty} (-1)^{j} \frac{t^{j}}{j!} 
			\exp\left( -\frac{\pi t j}{\sqrt{6N}} \right)\right) 
			\left(1 + O\left(N^{-\frac12}(\log N)^{4}\right)\right),\,\, \text{for } t\ge T_1. \label{eq-(4.12)}
		\end{align} 
		Note, by using Lemma \ref{lemma-3.4} and equation (\ref{eq-hardy}), we get 
		\begin{align}\label{eq-(4.3)}		
			p_{t}(N-t)=&p(N-t)\left(1+O\left(\frac{1}{\sqrt{\log N}}\right)\right)\\
			\notag
			=&p(N)\exp\left(-\frac{\pi t}{\sqrt{6N}}\right)\left(1+O\left(\frac{1}{\sqrt{\log N}}\right)\right),
		\end{align}
		for all $T_1\le t\le N^{\frac{3}{4}-\delta}=M$ with $0<\delta<\frac{1}{4}$.
		Now we will obtain the following lower bounds for $A_2$ and $A_1$.
		\begin{claim}\label{claim-1}
			\begin{align*}
				A_2\ge\left(2-\frac{2}{e}\right)\frac{p(N)^{2}}{\log N}\left(1+O\left(\frac{1}{\sqrt{\log N}}\right)\right).
			\end{align*}		
		\end{claim}	
		\begin{claim}\label{claim-2}
			\begin{align*}
				A_{1}\ge \frac{2p(N)^{2}}{e\log N}\left(1+O\left(\frac{1}{\sqrt{\log N}}\right)\right).
			\end{align*}
		\end{claim}
		First, we obtain the above lower bound for $A_2$.
		\begin{proof}[Proof of Claim \ref{claim-1}]
			
			We plug in $c_t(N)$ from (\ref{eq-(4.12)}) and $p_t(N-t)$ from (\ref{eq-(4.3)}) in the expression for $A_2$, and get
			\begin{align}\label{eq-(4.2)}
				\notag
				A_{2}\ge\; & p(N)\left( \sum_{t=T_2+1}^{N} p_{t}(N-t) 
				+\sum_{j=1}^{\infty} \sum_{t=T_2+1}^{N}p_{t}(N-t)   (-1)^{j} \frac{t^{j}}{j!} 
				\exp\left( -\frac{\pi t j}{\sqrt{6N}} \right)\right)\\
				\notag
				&\hspace{7.2cm}\times   \left(1 + O\left(N^{-\frac12}(\log N)^{4}\right)\right)\\
				\notag
				=& p(N)\left( \big( p(N) - p_{T_{2}}(N) \big) 
				+p(N)\sum_{j=1}^{\infty} \sum_{t=T_2+1}^{M}   (-1)^{j} \frac{t^{j}}{j!} 
				\exp\left( -\frac{\pi t (j+1)}{\sqrt{6N}} \right)\right)\\
				\notag
				&\hspace{7.2cm}\times\left(1+O\left(\frac{1}{\sqrt{\log N}}\right)\right)\\
				+&p(N)\left(\sum_{j=1}^{\infty} \sum_{t=M+1}^{N}p_{t}(N-t)   (-1)^{j} \frac{t^{j}}{j!} 
				\exp\left( -\frac{\pi t j}{\sqrt{6N}} \right)\right)   \left(1 + O\left(N^{-\frac12}(\log N)^{4}\right)\right).
			\end{align}
			Note
			\begin{align}\label{eq-A2error}
				\notag
				&\left|\sum_{j=1}^{\infty}\sum_{t=M+1}^{N}p_t(N-t)(-1)^{j} \frac{t^{j}}{j!} 
				\exp\left( -\frac{\pi t j}{\sqrt{6N}} \right)\right|\\
				\notag
				&\le\left(\sum_{t=M+1}^{N}p_t(N-t)\right)\max_{\substack{M+1 \le t \le N }}\left|\sum_{j=1}^{\infty}(-1)^{j} \frac{t^{j}}{j!} 
				\exp\left( -\frac{\pi t j}{\sqrt{6N}} \right)\right|\\
				\notag
				&\le p(N)\max_{\substack{M+1 \le t \le N }}\left(\exp \left(t\exp\left( -\frac{\pi t }{\sqrt{6N}} \right)\right)-1\right) \\
				&=O\left(\frac{p(N)}{(\log N)^{2}}\right).
			\end{align} 
			For the first part of (\ref{eq-(4.2)}), we use Lemma \ref{lemma-3.4} at $t=T_{2}$ to obtain
			\begin{equation}\label{eq-(4.4)}	
				p(N) - p_{T_2}(N) 
				= \frac{2p(N)}{\log N} \left(1+O\left(\frac{1}{\sqrt{\log N}}\right)\right).
			\end{equation}
			\newline
			The sum $ \sum_{t=T_2+1}^{M}$ in (\ref{eq-(4.2)}) can be simplified using Abel summation
			\begin{align}\label{eq-(4.5)}
				\notag
				&\sum_{T_2<t\le M}t^{j}\exp\left(-\frac{\pi t(j+1)}{\sqrt{6N}}\right)\\
				\notag
				=&M^{j}\frac{\exp\left(-\frac{\pi (j+1)}{\sqrt{6N}}\right)-\exp\left(-\frac{\pi (M+1)(j+1)}{\sqrt{6N}}\right)}{1-\exp\left(-\frac{\pi (j+1)}{\sqrt{6N}}\right)}\\
				\notag
				-&T_2^{j}\frac{\exp\left(-\frac{\pi (j+1)}{\sqrt{6N}}\right)-\exp\left(-\frac{\pi (T_2+1)(j+1)}{\sqrt{6N}}\right)}{1-\exp\left(-\frac{\pi (j+1)}{\sqrt{6N}}\right)}\\
				\notag
				-&j\int_{T_2}^{M}u^{j-1}\frac{\exp\left(-\frac{\pi (j+1)}{\sqrt{6N}}\right)-\exp\left(-\frac{\pi (u+1)(j+1)}{\sqrt{6N}}\right)}{1-\exp\left(-\frac{\pi (j+1)}{\sqrt{6N}}\right)}du\\
				=&T_2^{j}\frac{\exp\left(-\frac{\pi (T_2+1)(j+1)}{\sqrt{6N}}\right)}{1-\exp\left(-\frac{\pi (j+1)}{\sqrt{6N}}\right)}-M^{j}\frac{\exp\left(-\frac{\pi (M+1)(j+1)}{\sqrt{6N}}\right)}{1-\exp\left(-\frac{\pi (j+1)}{\sqrt{6N}}\right)}\\
				\notag
				+&\frac{j}{1-\exp\left(-\frac{\pi (j+1)}{\sqrt{6N}}\right)}\int_{T_2}^{M}u^{j-1}\exp\left(-\frac{\pi (u+1)(j+1)}{\sqrt{6N}}\right)du.
			\end{align}  
			Pluging in $T_2$ from equation (\ref{eq-T1T2}) in the first term of (\ref{eq-(4.5)}), we obtain    
			\begin{align}\label{eq-T2}
				T_2^{j}\frac{\exp\left(-\frac{\pi (T_2+1)(j+1)}{\sqrt{6N}}\right)}
				{1-\exp\left(-\frac{\pi (j+1)}{\sqrt{6N}}\right)}
				= \frac{2}{(j+1)\log N}\left(1 + O\left(\frac{\log \log N}{\log N}\right)\right).
			\end{align}
			The second term of (\ref{eq-(4.5)}), evaluated at $M = N^{\frac{3}{4}-\delta}$, gives  
			\begin{align}\label{eq-M}
				M^{j}\frac{\exp\left(-\frac{\pi (M+1)(j+1)}{\sqrt{6N}}\right)}
				{1-\exp\left(-\frac{\pi (j+1)}{\sqrt{6N}}\right)}
				\le \frac{C_{4}}{(j+1)N^{2}},
			\end{align}
			for some positive constant $C_{4}$.
			\newline
			Combining (\ref{eq-T2}), (\ref{eq-M}), we get the following expression for (\ref{eq-(4.5)}) 
			\begin{align}\label{eq-Error-j}
				\sum_{T_2 < t \le M} t^{j}\exp\left(-\frac{\pi t(j+1)}{\sqrt{6N}}\right)
				= \frac{2}{(j+1)\log N}\left(1 + O\left(\frac{\log\log N}{\log N}\right)\right)
				+ E_j,
			\end{align}
			and again applying integration by parts we estimate $E_j$ as  
			\[|E_j| \le \frac{L}{(j+1)(\log N)^{2}},\] 
			for some positive constant $L$ independent of $N$ and $j$.
			\newline       
			Now the estimates from (\ref{eq-A2error}), (\ref{eq-(4.4)}) and (\ref{eq-Error-j}) simplifies (\ref{eq-(4.2)}), and gives a lower bound for $A_{2}$
			\begin{align*}
				A_{2}\ge& \left(\frac{2p(N)^{2}}{\log N}+\frac{2p(N)^{2}}{\log N}\sum_{j=1}^{\infty}\frac{(-1)^{j}}{(j+1)!}\right)\left(1+O\left(\frac{1}{\sqrt{\log N}}\right)\right)\\
				\notag
				=&\left(2-\frac{2}{e}\right)\frac{p(N)^{2}}{\log N}\left(1+O\left(\frac{1}{\sqrt{\log N}}\right)\right).
			\end{align*}  
		\end{proof} 
		Next, we calculate the lower bound for $A_{1}$.
		\begin{proof}[Proof of Claim \ref{claim-2}]
			
			We again use (\ref{eq-(4.1)}) and (\ref{eq-(4.3)}), which are also valid for $T_{1} \leq t \leq T_{2}$, and obtain:		
				\begin{align}\label{eq-(4.7)}
					\notag
					A_1&=\sum_{t=T_1+1}^{T_2}c_{t}(N)p_{t}(N-t)\\
					\notag
					&=\left(\sum_{t=T_1+1}^{T_2}c_{t}(N)p(N-t)\right)\left(1+O\left(\frac{1}{\sqrt{\log N}}\right)\right) \\
					&\ge p(N)^{2}\left(\sum_{t=T_1+1}^{T_2}
					\exp\left(-t\exp\left(-\tfrac{\pi t}{\sqrt{6N}}\right)\right)
					\exp\left(-\frac{\pi t}{\sqrt{6N}}\right)\right)\left(1+O\left(\frac{1}{\sqrt{\log N}}\right)\right).
				\end{align}
			Let
			$$
			G(t)=\exp\left(-t\exp\left(-\frac{\pi t}{\sqrt{6N}}\right)\right)
			\exp\left(-\frac{\pi t}{\sqrt{6N}}\right).
			$$
			Using Euler summation, we obtain 
			\begin{align}\label{eq-(4.8)}
				\notag
				\sum_{T_1<t\le T_2} G(t)
				&= \int_{T_1}^{T_2} G(t)\,dt 
				+ \int_{T_1}^{T_2} (t-[t])\,G'(t)\,dt\\& + G(T_2)([T_2]-T_2) - G(T_1)([T_1]-T_1). 
			\end{align}
			After substituting the explicit values of $T_1$ and $T_2$ into the second term of the above expression, we get
			\begin{align*}
				\int_{T_1}^{T_2} (t-[t])\,G'(t)\,dt
				+ G(T_2)([T_2]-T_2)
				- G(T_1)([T_1]-T_1)
				= O\!\left(N^{-\frac{1}{2}}\right).
			\end{align*}
			Choosing
			$$
			u(t):=\exp\left(-\frac{\pi t}{\sqrt{6N}}\right),
			$$
			we have
			$$
			I=\int_{T_1}^{T_2}G(t)\,dt=\frac{\sqrt{6N}}{\pi}\int_{u(T_2)}^{u(T_1)}\exp\left(\frac{\sqrt{6N}}{\pi}u\log u\right)\,du.
			$$
			Let
			$$
			\lambda=\frac{\sqrt{6N}}{\pi} \quad \text{ and } \quad F(u)=-u\log u.
			$$
			Since $F'(u)>0$ for $u\in[u(T_2),u(T_1)]$, the function $F(u)$ is increasing in this interval, and thus the integral dominates near the lower endpoint $u(T_{2})$. By the Laplace method for endpoint maxima \cite{lapinski2019multivariate}, we obtain 
				\begin{align}
					\notag
					I &= \frac{\sqrt{6N}}{\pi}\,
					\frac{\exp\left(-\lambda F(u(T_2))\right)}{\lambda F'(u(T_2))}\left(1+O\left(\frac{ F'(u(T_2))}{\lambda}\right)\right) \\
					\label{eq-(4.9)}
					&= \frac{2}{e\log N}\left(1+O\left(\frac{\log\log N}{\log N}\right)\right).
				\end{align}
			Now (\ref{eq-(4.9)}) gives an asymptotic formula for $\sum_{T_1<t\le T_2}$ in (\ref{eq-(4.8)}). Substituting this into (\ref{eq-(4.8)}), we obtain
			\begin{equation*}
				A_{1}\ge \frac{2p(N)^{2}}{e\log N}\left(1+O\left(\frac{1}{\sqrt{\log N}}\right)\right).  
			\end{equation*}
		\end{proof}
		Finally, we obtain our required lower bound for $Z(N)$ from (\ref{eq-(A)}), using Claim \ref{claim-1} and Claim \ref{claim-2} 
		$$
		Z(N) \ge \frac{2\, p(N)^{2}}{\log N} \left(1+O\left(\frac{1}{\sqrt{\log N}}\right)\right).
		$$
	\end{proof}
	\begin{remark}
		The contribution from $\sum_{t=1}^{T_{1}}c_{t}p_{t}(N-t)$ is very small, of order $O\left(\frac{p(N)^{2}}{(\log N)^{2}}\right)$, and hence does not improve our lower bound for $Z(N)$. 
	\end{remark}
	
	\begin{proof}[Proof of Theorem \ref{theorem-1.3}] 
		By the Murnaghan-Nakayama rule (Theorem \ref{theorem-2.1}) and Lemma \ref{lemma-2.2},
		\begin{equation*} \label{eq-(3)}
			Z_t(N) \geq c_t(N)\, p(N - t).  
		\end{equation*} 
		From (\ref{eq-hardy}), the partition function $p(N - t)$  is given by  
		\begin{equation*}\label{eq-(3.3)}
			p(N - t) = \frac{1}{4(N - t)\sqrt{3}} \exp\left( \frac{2\pi}{\sqrt{6}} \sqrt{N - t} \right) \left(1 + O\left((N - t)^{-1/2}\right)\right).  
		\end{equation*} 
		Combining these bounds with Proposition \ref{proposition-3.2}, we complete the proof of Theorem \ref{theorem-1.3}.  
	\end{proof}

	\bibliographystyle{abbrv}					
	\bibliography{report}  				
\end{document}